\documentclass [12pt]{amsart}

\usepackage[left=1in,right=1in,top=1in,bottom=1in]{geometry} 

\usepackage{amsmath}

\usepackage{url}

\usepackage{amssymb,amsthm,amsmath}

\usepackage{ifthen}

\usepackage{graphicx}

\DeclareGraphicsExtensions{.eps}

\newtheorem{theorem}{Theorem}[section]

\newtheorem{proposition}[theorem]{Proposition}

\theoremstyle{definition}

\newtheorem{definition}[theorem]{Definition}

\newtheorem{example}[theorem]{Example}

\newtheorem{remark}[theorem]{Remark}

\numberwithin{equation}{section}

\begin{document}
	
	\date{\today}

	\title[Regularity Of The Semi-group Of Regular Probability Measures]{ Regularity Of The Semi-group Of Regular Probability Measures On Compact Hausdorff  Topological Groups}	
\author{M. N. N. Namboodiri}
\address{Department of  Mathematics, Cochin University of Science \& Technology, Kochi, Kerala, India-682022}
\email{mnnadri@gmail.com}

	\thanks{The author is supported by the KSCSTE Emeritus Scheme.}
	
	\date{ABC XX, XXXX and, in revised form, XYZ XX, XXXX.}


\begin{abstract}
	Let $G$ be a compact group, and $P(G)$ denotes the class of all regular probability measures on $G$. It is well known that $P(G)$ is a semi-group under the convolution of measures. This semi-group has been studied elaborately and intensely by 
There are many deep results on the structure of regular probability measures $P(G)$ on compact/locally compact, Hausdorff topological groups $G$.
See, for instance, the classic monographs by KR Parthasarathy \cite{KRP}, Ulf Grenander \cite{UFG3}A. Mukherjea and Nicolas A.Tserpes \cite{MNT}, Wendel \cite{WEN} to quote selected references. In his remarkable paper, Wendel proved many deep theorems in this context. He proved that $P(G)$  is a semi-group which is not a group, by proving that the only invertible elements are point mass supported measures (Dirac delta measures).

In this short paper, we prove that $P(G)$  \textbf{is not algebraically regular } in the sense that not every element has a generalized inverse. However, we prove that $P(G)$  can be embedded in larger concrete algebraically regular semi-groups. Also, an attempt is made to identify algebraically regular elements in some special cases.
\end{abstract}
\maketitle
\section{Introduction}
As mentioned in the abstract, it is well known that the 
set $P(G)$ of regular probability measures on a topological group $G$ is a semi-group under convolution, which is abelian if and only if the group $G$ is abelian. It is also known that $P(G)$ is a compact convex set under the weak$^{*}$ topology of measures. Wendel \cite{WEN} in his remarkable paper, established many significant results regarding the algebraic, topological as well as geometric structure of $P(G)$. He showed that $P(G)$ is a closed convex semi-group which is not a  group except for trivial group $\{e\}$ by showing that the only invertible elements are point mass measures supported on single elements.

The problem we consider is  \textbf{the algebraic regularity} of $P(G)$. A semi-group is called \textbf{algebraically regular} if each of its elements has a generalized inverse.

The main theorem proved in this article is Theorem \ref{mainres}, which states that $P(G)$ is not a regular semi-group unless, of course, for the trivial case $G=\{e\}$. In section \ref{sec4}, the embedding of this semi-group into a regular one is considered. In the concluding section \ref{sec5}, several related problems are given, such as the optimality of this embedding. However, a possible groundwork is prepared using the already existing theory of non-commutative Fourier transform of measures 
in  $P(G)$ for the special case where $G$ is a compact Lie group \cite{DAM}.
\section{Priliminaries}\label{sec2}
Let $G$ be a compact, Hausdorff topological group and $\mathcal{B}$ denote the $\sigma$-algebra of all Borel sets in $G$. A probability measure $\mu$ is a nonnegative countably additive function on $\mathcal{B}$ such that the total mass $\mu(G)=1$. A point mass measure or Dirac delta measure is a measure $\mu$ for which there is an element $x\in G$ such that $\mu(A)=1$ if $x\in A $ and zero otherwise; $A\in \mathcal{B}$. Such a measure is usually denoted as $\delta_{x}$. One of the interesting results of Wendel is that the only invertible elements in $P(G)$ are Dirac delta measures. The product in $P(G)$ is the convolution  $\star$ which is defined as follows;

\begin{definition}(\textbf{Convolution})
	Let $\mu,\nu \in P(G)$. Then $\mu \star \nu$ is the probability measure defined as $\mu \star\nu (A) =\int\mu(Ax^{-1})d \nu(x)$ 
for every $ A \in \mathcal{B}$.
\end{definition}
\begin{definition}
(\textbf{Generalised inverse})
Let $\mathcal {S}$ be a semigroup and let $s\in \mathcal {S}$. An element $s^{\dagger} \in \mathcal{S}$ is called a generalised inverse of $s$ if $ss^{\dagger} s=s.$
\end{definition}
	For example, it is well known that the set $M_{N}(\mathbb{C})$ of all complex matrices of finite order $N$ is a regular semi-group. The property of algebraic regularity is almost essential in the fundamental characterization theorems of KSS Nambooripad \cite{KSS}. In this short note, we do not analyze the implications and consequences of Nambooripad's theory in this concrete semi-group which is postponed to a different project altogether.

\section{Regularity Question}\label{sec3}
 For a compact topological group $G$, J.G. Wendel \cite{WEN} proved that \textbf{the set $P(G)$ is a semi-group which is not a group under convolution by proving that the only invertible elements in $P(G)$ are Dirac delta measures.} One crucial property needed for measures under consideration is the \textbf{regularity} which is not guaranteed for compact topological groups. Next, we quote a fundamental theorem due to Wendel.
\begin{definition}
(\textbf{Support}) Support of $\mu \in P(G)$ is defined as 
$supp(\mu)= \{g \in G: \mu(E_{g})>0 \textrm{ for every neighbourhood } E_{g} \textrm{ of }g \in G\}.$
\end{definition}

\begin{theorem}[Wendel] Let $A$ and $B$ be supports of two measures $\mu$ and $\nu$     in $P(G)$.Then 	$supp(\mu \star \nu)= AB=\{xy| x\in A, y \in B\}$
\end{theorem}
Now we prove the main theorem of this short research article.	
\begin{theorem}\label{mainres}		Let $G$	 be a nontrivial compact topological group. Then $P(G)$
is not regular.
\end{theorem}
\begin{proof}
First we prove the assertion for the special case for which group $G$ is such that $a^{2}\neq e$ for some $a \in G$.
Let $a \in G$ be such that $a^{2} \neq e$.Consider the probability measure $\mu=\frac{\delta_{e}+\delta_{a}}{2}$ where $\delta_{x}$ is the Dirac delta measure at $x$ for each $x \in G$.We show that $\mu$ does not have a generalised inverse. Let if possible a generalised inverse $\mu^{\dag}$ of $\mu$ exist. Therefore we have
\begin{equation}\label{eq1}
	\mu\star \mu^{\dag} \star \mu=\mu.
\end{equation}	
and $\mu \star \mu^{\dag} $  is an idempotent. Clearly $ supp(\mu)=\{e,a\}$ and  $H$= $supp(\mu \star \mu^{\dag})$ is a compact subgroup of $G$ by Theorem $1$ in \cite{DAM}. Now combining
Wendels's theorem and equation \ref{eq1} we find that 
\begin{equation} \label{eq2}
	H.\{e,a\}=\{e,a\}
\end{equation}	
Let $h\in H$ and the  equation \ref{eq2} above  implies that 
\begin{equation}
	h.\{e,a\} \subset \{e,a\}\Rightarrow he=e \quad or \quad he=a 	
	\textrm{ and } \quad ha=e \quad or ha=a.
\end{equation}	
Now $he=e\Rightarrow h= e\textrm{ or } h=a$.
Again $ha=e \Rightarrow h=a^{-1} \textrm{ or }  ha=a  \Rightarrow h=e$.

Thus to summarise we get $h=e \quad or \quad h=a^{-1}$.
Thus the possibilities are $h=e$  for all $h,\{h=e,h=a^{-1}\}$,$\{h=e,h=a\} $. Thus we get $H=\{e\}$ or $H=\{e,a\}$ or $H=\{e,a^{-1}\}$. Now $H$ is a group. The second and third option would imply that $a^{2}=e$ which is against the hypothesis.Therefore
we have $H=\{e\}$. Now $\mu\star \mu^{\dag}$ is a projection and therefore we get $\mu\star \mu^{\dag}=\delta_{e}$, which is the identity.Thus $ \mu$ is right invertible
Let $supp(\mu^{\dag}) =F$. Observe that   $ supp(\mu\star \mu^{\dag})=\{e\}$. Therefore we have
$	\{e,a\}.F=\{e\}.$
Let $f\in F$.Then $e.f=f=e$ and $a.f=e$ $\Rightarrow a=e$, which is again not possible. All these absurd conclusions are consequence of the assumption that $\mu$ is regular.

Now let $G$ be such that $a^{2}=e$ for every $a\in G$. Let $a\neq e$. Consider $\mu=\alpha_{0}\delta_{e}+\alpha_{1}\delta_{a}$,where $0\leq\alpha_{0},\alpha_{1}\leq 1,\alpha_{0}+\alpha_{1}=1$.Then we have $\mu\in P(G)$ and $supp(\mu)=\{e,a\}$. First we show that $\mu$ is an idempotent if and only if $\alpha_{0}=\alpha_{1}=\frac {1}{2}$.

Observe that $	\mu^{2}=(\alpha_{0}^{2} +\alpha_{1}^{2}) \delta_{e}+2\alpha_{0} \alpha_{1} \delta_{a}.$
Therefore $\mu^{2}=\mu$ if and only if $	\alpha_{0}^{2}+\alpha_{1}^{2}=\alpha_{0}$         
	  and $2\alpha_{0}\alpha_{1}=\alpha_{1},$
if and only if $
\alpha_{0}=\alpha_{1}=\frac{1}{2}.$
Now, let if possible, $\mu$	for which $\alpha_{0} \neq \frac{1}{2}$ has a generalised inverse  $\mu^{\dagger}$. It is an easy consequence of Wendel's support theorem that 
$\mu^{\dagger}=\beta_{0}\delta_{e}+\beta_{1}\delta_{a}$ where $0\leq \beta_{0},\beta_{1} \leq 1,\beta_{0}+\beta_{1}=1$; the proof is as follows. Let $H= supp(\mu^{\dagger})$. We have by Wendel's theorem
\begin{equation}
	\{e,a\}.H.\{e,a\}= \{e,a\}
\end{equation}	

Let $h\in H$. Then $h\in\{e,a\}$.Thus  $supp(\mu^{\dagger})\subseteq\{e,a\}$.

Now we have that  $\mu \star \mu^{\dagger}$ is an idempotent. But an easy computation shows that 
\begin{equation}
\mu\star\mu^{\dagger}=(\alpha_{0}\beta_{0}+\alpha_{1}\beta_{1})	\delta_{e}+(\alpha_{0}\beta_{1}+\alpha_{1}\beta_{0})	\delta_{a}
\end{equation}	
Therefore we must have $(\alpha_{0}\beta_{0}+\alpha_{1}\beta_{1})=\frac{1}{2}=(\alpha_{0}\beta_{1}+\alpha_{1}\beta_{0}).$
Solving the above linear equations we obtain $\alpha_{0}=\frac{1}{2}=\alpha_{1}$ provided $\beta_{0}\neq \beta_{1}$.
Now assume that $\beta_{0}=\beta_{1}$.This means that $\beta_{k}=\frac{1}{2}$ for all $k$. Now we use the full force of generalised inverse as follows. We have
\begin{equation}\nonumber
[(\alpha_{0}\beta_{0}+\alpha_{1}\beta_{1})	\delta_{e})+(\alpha_{0}\beta_{1}+\alpha_{1}\beta_{0})	\delta_{a}]\times (\alpha_{0}\delta_{e}+\alpha_{1}\delta_{a})=	\alpha_{0}\delta_{e}+\alpha_{1}\delta_{a}
	\end{equation}
\begin{equation}\nonumber
\Rightarrow \frac{\alpha_{0}+\alpha_{1} }{2}=\alpha_{0} \Rightarrow \alpha_{0}=\frac{1}{2}=\alpha_{1}
\end{equation}
Therefore for $\alpha_{k}\neq\frac{1}{2},0\leq \alpha_{0},\alpha_{1}\leq 1,\alpha_{0}+\alpha_{1}=1$ $\alpha_{0}\delta_{e}+\alpha_{1}\delta_{a}$ will not be regular.This completes the proof.
\qquad \qquad \qquad \qquad \qquad  \qquad \qquad $\Box$
\end{proof}
\begin{proposition}
Let $G$	be a compact topological group.For $g\in G$ let $\mu=\frac{\delta_{e}+\delta_{g}}{2}$.Then $\mu$ is regular if and only if $g^{2}=e$
\end{proposition}	
\begin{proof}
Observe that the proof of theorem 3.3 above essentially establishes the assertion.	
\end{proof}	
\begin{remark}
The above regularity problem was stated and left open in \cite{MST}. Wendel proved that the only invertible elements are Dirac delta measures at various points. The problem of characterizing regular elements of $P(G)$ seems interesting. We do not address this problem here. Observe that towards the end of the proof of the above theorem, we actually solved this question for a very special case for which $G=\{e,a\}$.
In fact we prove that the only regular elements of $P(G)$ are $\{\delta_{e},\delta_{a},\frac {\delta_{e}+\delta_{a}}{2}\}$.
\end{remark}
\begin{remark}
The set $P(G)$ is a closed convex set under weak$^{*}$ topology of measures, and $\{\delta_{g}:g \in G\}$ is the set of extreme points of $P(G)$. Hence by Krein-Millmann theorem, the closed convex hull $\overline{conv}{\{\delta_{g}:g \in G}\}=P(G)$. In particular, if one considers the subsemi-group ${conv}{\{\delta_{g}:g \in G}\}$, it may be possible to locate all regular elements in it geometrically. This possibility is under investigation.
\end{remark}
\begin{remark}
In a general semigroup $\Omega$	if $\omega\in \Omega$ has a generalised inverse ,then it has a Moore -Penrose invese: to be explicit if $\omega g \omega=g $ for some $g\in \Omega$ then 
\begin{equation}
	\omega \star\omega^{\dagger}\star \omega=\omega  \quad \& 
	\end{equation}
\begin{equation}
	\omega^{\dagger}\star\omega \star\omega^{\dagger}
	\end{equation}
where $\omega^{\dagger}=\omega g \omega$. 

So to characterize generalized invertibility, it will be enough to characterize Moore-Penrose invertibility. So in the following example, we try to identify \textbf{Moore-Penrose invertible elements.}
\end{remark}	
\begin{example}
Let $G$	be a compact topological group such that $g^{2}=e$ for all $g\in G$.Let $S$ be the sub semi group given by
\begin{equation}
S=Conv\{\delta_{g}:g\in G\}	
\end{equation} where 'Conv'	denotes the convex hull .For a finite set $\{g_{1},g_{2},...g_{n}\}\subset G$ let
\begin{equation}
\mu=\frac{\delta_{e}+\delta_{g_{1}}+\delta_{g_{2}}+...+\delta_{g_{n}}}{n+1}	
\end{equation}	
Then $\mu$ is regular.
\begin{proof}
	Of course, one can directly prove that $\mu$ is regular by brutal computation. However, our main interest being the characterization of regular elements, we give a systematic way of arriving at regular elements, $\mu$ being one of them. To start with, we assume that;
	
	\begin{equation}
	 \mu=\Sigma_{k=0}^{n} \alpha_{k}\delta_{g_{k}}, \alpha_{k}  > 0,\& \Sigma_{k=0}^{n}\alpha_{k}=1 
	\end{equation}
	First we show that if $\mu^{\dagger}=\Sigma_{j=1}^{m}\beta_{j} \delta_{h_{j}}$ is the Moore-Penrose  inverse implies that $\{h_{j},j=1,2,..m.\}=\{g_{j},j=1,2,...n\}.$
	
	If $\gamma$ is a generalised inverse of $\mu$,then we will have 
	\begin{equation}
	\mu \star \mu^{\dagger} \star \mu=\mu	,            \quad \& 
	\end{equation}
	\begin{equation}
	\mu^{\dagger}\star \mu \star \mu^{\dagger}=\mu^{\dagger}
	\end{equation}	
Hence by Wendel's support theorem we have 
\begin{equation} 
	\{g_{k}:k=0,1,2,...n\}\{h_{k}:k=0,1,2,...m\}\{g_{k}:k=0,1,2,...n\}=\{g_{k}:k=0,1,2,...n.\}
	\end{equation}
In particular we have
\begin{equation}
\{h_{k}:k=0,1,2,..m\}	\subset \{g_{k}:k=0,1,2,...n.\}.
\end{equation}
Similar argument by using equation $3.12$ above we have
\begin{equation}
	\{h_{k}:k=0,1,2,...m\}\{g_{k}:k=0,1,2,...n\}\{h_{k}:k=0,1,2,...m\}=\{h_{k}:k=0,1,2,...m\}
\end{equation}
Since $G$ is abelian and by using tha fact that $h^{2}_{k}=e$,we obtain the reverse inclusion namely $ \{g_{k}:k=0,1,2,...n.\}\subset  \{g_{k}:k=0,1,2,...n.\}.$.
This proves our claim.Therefore we may assume that $\mu^{\dagger}$ is the  generalised inverse of $\mu$ implies that 
\begin{equation}
\mu^{\dagger}=\Sigma_{k=0}^{n}	\beta_{k}\delta_{g_{k}}
\end{equation}	
where $\beta_{k} > 0  \quad \&  \quad \Sigma_{k=0}^{n}\beta_{k}=1$.
It iseasy to see that 
\begin{equation}
\mu \star \mu^{\dagger}	=\Sigma_{j=0}^{n} (\Sigma_{g_{k}g_{l}=g_{j}}\alpha_{k}\beta_{l})\delta_{g_{j}}.
\end{equation}

	and
\begin{equation}
\mu \star\mu^{\dagger} \star \mu=	[\Sigma_{j=0}^{n}\sigma_{j}\delta_{g_{j}}]\star \mu=\Sigma_{k=0}^{n} \alpha_{k}\delta_{g_{k}}
\end{equation}	
where $\sigma_{j}=\Sigma_{g_{k}g_{l}=g_{j}}\alpha_{k}\beta_{l}$	for each $j$.
Therefore we find that
\begin{equation}
	\Sigma_{j=0}^{n} (\Sigma_{g_{k}g_{l}=g_{j}}\sigma_{k}\alpha_{l})\delta_{g_{j}}=\Sigma_{j=0}^{n} \alpha_{j}\delta_{g_{j}}
\end{equation}	
Therefore we have that 
\begin{equation}
\Sigma_{g_{k}g_{l}=g_{j}}\sigma_{k}\alpha_{l}= \alpha_{j}	
\end{equation}	
for every $j$.	Substituting terms we get
\begin{equation}
\Sigma_{g_{k}g_{l}=g_{j}}(\Sigma_{g_{i}g_{l}=g_{k}}\alpha_{i}\beta_{l})\alpha_{l}= \alpha_{j}	
\end{equation}	
for$j=0,1,2,...n.$.We consider equation $3.15$ which can be written as 
\begin{equation}
\Sigma_{k=0}^{n} \sigma_{k}\alpha_{S_{j}(k)}	
\end{equation}	
for each $j$, $S_{j}$ is the permutation on $\{0,1,2,...n.\}$ given by $g_{j}g_{k}\rightarrow g_{S_{j}(k)}$,$j,k\in\{0,1,2,...,n\}$.This can again be written as a matrix equation as follows:
\begin{equation}
\textbf{A} =\begin{bmatrix}
\alpha_{S_{(0)}(0)}\,\alpha_{S_{(0)}(1)},\quad	\ldots \alpha_{S_{(0)}(n)} \\
\alpha_{S_{(1)}(0)}\,\alpha_{S_{(1)}(1)},\quad	\ldots \alpha_{S_{(1)}(n)}\\
.\\      \quad  .\\

.         \quad .,           .      \quad .                              \quad       .

\\

\alpha_{S_{(n)}(0)},\alpha_{S_{(n)}(1)},\quad	\ldots \alpha_{S_{()}(n)}
\end{bmatrix}
\end{equation}
and the corresponding equation is as follows:
\begin{equation}
	\begin{bmatrix}
		\alpha_{S_{(0)}(0)}\,\alpha_{S_{(0)}(1)},\quad	\ldots \alpha_{S_{(0)}(n)} \\
		\alpha_{S_{(1)}(0)}\,\alpha_{S_{(1)}(1)},\quad	\ldots \alpha_{S_{(1)}(n)}\\
		.\\      \quad  .\\
		
		.         \quad .,           .      \quad .                              \quad       .
		
		\\
		
		\alpha_{S_{(n)}(0)},\alpha_{S_{(n)}(1)},\quad	\ldots \alpha_{S_{(n)}(n)}
\end{bmatrix}
\begin{bmatrix}
\sigma_{0}\\

\sigma_{1}\\

\vdots\\
\sigma_{k}\\
\vdots\\

\sigma_{n}	
\end{bmatrix}
=\begin{bmatrix}
\alpha_{0}	\\
\alpha_{1}\\

\vdots\\
\alpha_{k}\\
\vdots\\
\alpha_{n}
\end{bmatrix}	
\end{equation}	
Now equation$3.13$ can be explicitely written as follows:
\begin{equation}
		\begin{bmatrix}
		\alpha_{S_{(0)}(0)}\,\alpha_{S_{(0)}(1)},\quad	\ldots \alpha_{S_{(0)}(n)} \\
		\alpha_{S_{(1)}(0)}\,\alpha_{S_{(1)}(1)},\quad	\ldots \alpha_{S_{(1)}(n)}\\
		.\\      \quad  .\\
		
		.         \quad .,           .      \quad .                              \quad       .
		
		\\
		
		\alpha_{S_{(n)}(0)},\alpha_{S_{(n)}(1)},\quad	\ldots \alpha_{S_{(n)}(n)}
	\end{bmatrix}
	\begin{bmatrix}
		\beta_{0}\\

		\beta_{1}\\

		\vdots\\
		\beta_{k}\\
		\vdots\\
		
		\alpha_{n}	
	\end{bmatrix}
	=\begin{bmatrix}
		\sigma_{0}	\\
		\sigma_{1}\\
		
		\vdots\\
		\sigma_{k}\\
		\vdots\\
		\sigma_{n}
	\end{bmatrix}	
\end{equation}
Cmbining equations $3.20 \& 3.19$ we find that the setermining equation is as follows:
\begin{equation}
	{\begin{bmatrix}
		\alpha_{S_{(0)}(0)}\,\alpha_{S_{(0)}(1)},\quad	\ldots \alpha_{S_{(0)}(n)} \\
		\alpha_{S_{(1)}(0)}\,\alpha_{S_{(1)}(1)},\quad	\ldots \alpha_{S_{(1)}(n)}\\
		.\\      \quad  .\\
		
		.         \quad .,           .      \quad .                              \quad       .
		
		\\
		
		\alpha_{S_{(n)}(0)},\alpha_{S_{(n)}(1)},\quad	\ldots \alpha_{S_{(n)}(n)}
	\end{bmatrix}}^{2}
	\begin{bmatrix}
		\beta_{0}\\

		\beta_{1}\\

		\vdots\\
		\beta_{k}\\
		\vdots\\
		
		\beta_{n}	
	\end{bmatrix}
	=\begin{bmatrix}
		\alpha_{0}	\\
		\alpha_{1}\\
		
		\vdots\\
		\alpha_{k}\\
		\vdots\\
		\alpha_{n}
	\end{bmatrix}	
\end{equation}	
\end{proof}	
\begin{remark}
Equations $3.19 \& 3.20$	above determines all regular elements in the semigroup $S,(3.6)$.In particular it follows that the middle points $\{ \frac {1}{k} \Sigma_{j=0}^{k}\delta_{g_{j}}: k=2,3,...,n+1\}$ are all regular.Existance of other noninvertible regular elements needs analysis of the matrix equation $3.19$.

Observe that the permutation $S_{0}$ is the identity.Moreover we have that $S_{j}(j)=e$ for all $j=0,1,2,...,n.$.Therefore the diagonal entry of the matrix $A$ is the same namely $\alpha_{0}$.
 \textbf{Obstructions:}

 In what follows we assume that $\alpha_{k}>0 , k=0,1,2,...n  \quad \& \quad  \Sigma_{k=0}^{n}\alpha_{k}=1$.If $A$ is  invertible ,then equation $3.25$ will have the unique solution namely
 \begin{align*}
 	\sigma_{k}=1  ,k=0 \quad  \&   \quad \\
 	             =0,k=1,2,..,n.
 \end{align*}	

Hence equation 3.25 becomes
\begin{equation}
			\begin{bmatrix}
		\alpha_{S_{(0)}(0)}\,\alpha_{S_{(0)}(1)},\quad	\ldots \alpha_{S_{(0)}(n)} \\
		\alpha_{S_{(1)}(0)}\,\alpha_{S_{(1)}(1)},\quad	\ldots \alpha_{S_{(1)}(n)}\\
		.\\      \quad  .\\
		
		.         \quad .,           .      \quad .                              \quad       .
		
		\\
		
		\alpha_{S_{(n)}(0)},\alpha_{S_{(n)}(1)},\quad	\ldots \alpha_{S_{(n       )}(n)}
	\end{bmatrix}
	\begin{bmatrix}
		\beta_{0}\\

		\beta_{1}\\

		\vdots\\
		\beta_{k}\\
		\vdots\\
		
		\beta_{n}	
	\end{bmatrix}
	=\begin{bmatrix}
		1	\\
		0\\
		
		\vdots\\
		0\\
		\vdots\\
		0
		\end{bmatrix}
\end{equation}
Therefore we will have $\beta_{k}=0$ for all $k=1,2,...,n$  which implies that $\alpha_{0}\beta_{0}=1$.But this means that $\alpha_{0}=1=\beta_{0}$.Hence we get the cotradictory implication that $\alpha_{k}=0,k=1,2,...,n.$.

\begin{remark}
\textbf{Thus for any set $\{\alpha_{k}:\alpha_{k}>0 \quad \& \quad \Sigma_{k=0}^{n}=1\}$ for which the corresponding matrix $A$ is invertible the probability measure $\mu= \Sigma_{k=0}^{n}	\alpha_{k}\delta_{g_{k}}$ will not have a generalised inverse in the semigroup $S$ given at the beginning of this example.}

\textbf{Now the diagonal dominance  is a verifiable condition that implies invertibility.Since $\alpha_{S_{k}(k)}=\alpha_{0}$ for all $k$, the above condion reduces to the following inequalityi given below:}
\begin{equation}
	\alpha_{0}> \frac{n}{n+1}.
	\end{equation}

However the special case for which $\alpha_{k}=\frac{1}{n+1}$,the corresponding probability measure $\mu=\frac{\Sigma_{k=0}^{n}\delta_{g_{k}}}{n+1}$ will have the generalised inverse namely
$\mu$ itself which is the \textbf{midpoint  of the convex polytope}. 	
\end{remark}
\begin{example}
The case $n=1$ has already been done.Now consider the case $n=2$ so that $G=\{e,g_{1},g_{2} : g^{2}_{k}=e\}$.	In this case we will have $g_{1}g_{2}=e,or g_{1} or g_{2}$,which implies that $G$ is a two element one.

Now we consider the case $n=3$.We prove the following relations namely
\begin{equation}
g_{1}g_{2}=g_{3},g_{2}g_{3}=g_{1} \& g_{1}g_{3}=g_{2}.	
\end{equation}	
The above relations determines the permutations $S_{0},S_{1},S_{2} \&S_{3}$ on $\{0,1,2,3.\}$.Simple computations reveals that the matrix $\textbf{A}$ is as follows:
\begin{equation}
	\textbf{A}= \begin{bmatrix}
		\alpha_{0}\quad \alpha_{1}\quad \alpha_{2} \quad \alpha_{3}\\
		\alpha_{1}\quad \alpha_{0}\quad \alpha_{3} \quad \alpha_{2}\\
		\alpha_{2}\quad \alpha_{3}\quad \alpha_{0} \quad \alpha_{1}\\
		\alpha_{3}\quad \alpha_{2}\quad \alpha_{1} \quad \alpha_{0}
		\end{bmatrix}
\end{equation}
	which is a Hermetian doubly stochastic matrix.	
	Now the equation is 
	\begin{equation}
		\begin{bmatrix}
			\alpha_{0} \quad \alpha_{1}\quad \alpha_{2}\quad\alpha_{3}\\   	                                                                                             
			\alpha_{1}\quad \alpha_{0}\quad \alpha_{3} \quad \alpha_{2}\\  
			\alpha_{2}\quad \alpha_{3}\quad \alpha_{0} \quad \alpha_{1}\\
			\alpha_{3}\quad \alpha_{2}\quad \alpha_{1} \quad \alpha_{0}
		\end{bmatrix}
	\begin{bmatrix}
		\sigma_{0}\\
		\sigma_{1}\\
		\sigma_{2}\\
		\sigma_{}3
	\end{bmatrix}
=
\begin{bmatrix}
	\alpha_{0}\\
\alpha_{1}\\
\alpha_{2}\\
\alpha_{3}
\end{bmatrix}	
	\end{equation}.
In addition we have subsequent equation 
\begin{equation}
	\begin{bmatrix}
		\alpha_{0} \quad \alpha_{1}\quad \alpha_{2}\quad\alpha_{3}\\   	                                                                                             
		\alpha_{1}\quad \alpha_{0}\quad \alpha_{3} \quad \alpha_{2}\\  
		\alpha_{2}\quad \alpha_{3}\quad \alpha_{0} \quad \alpha_{1}\\
		\alpha_{3}\quad \alpha_{2}\quad \alpha_{1} \quad \alpha_{0}
	\end{bmatrix}
	\begin{bmatrix}
		\beta_{0}\\
		\beta_{1}\\
		\beta_{2}\\
		\beta_{3}
	\end{bmatrix}
	=
	\begin{bmatrix}
		\sigma_{0}\\
		\sigma_{1}\\
		\sigma_{2}\\
		\sigma_{3}
	\end{bmatrix}	
\end{equation}.
The combined equation as before becomes
	\begin{equation}
	\begin{bmatrix}
		\alpha_{0} \quad \alpha_{1}\quad \alpha_{2}\quad\alpha_{3}\\   	                                                                                             
		\alpha_{1}\quad \alpha_{0}\quad \alpha_{3} \quad \alpha_{2}\\  
		\alpha_{2}\quad \alpha_{3}\quad \alpha_{0} \quad \alpha_{1}\\
		\alpha_{3}\quad \alpha_{2}\quad \alpha_{1} \quad \alpha_{0}
	\end{bmatrix}
	\begin{bmatrix}
		\sigma_{0}\\
		\sigma_{1}\\
		\sigma_{2}\\
		\sigma_{}3
	\end{bmatrix}
	=
	\begin{bmatrix}
		\alpha_{0}\\
		\alpha_{1}\\
		\alpha_{2}\\
		\alpha_{3}
	\end{bmatrix}	
\end{equation}.
In addition we have subsequent equation 
\begin{equation}
	{\begin{bmatrix}
		\alpha_{0} \quad \alpha_{1}\quad \alpha_{2}\quad\alpha_{3}\\   	                                                                                             
		\alpha_{1}\quad \alpha_{0}\quad \alpha_{3} \quad \alpha_{2}\\  
		\alpha_{2}\quad \alpha_{3}\quad \alpha_{0} \quad \alpha_{1}\\
		\alpha_{3}\quad \alpha_{2}\quad \alpha_{1} \quad \alpha_{0}
	\end{bmatrix}}^{2}
	\begin{bmatrix}
		\beta_{0}\\
		\beta_{1}\\
		\beta_{2}\\
		\beta_{3}
	\end{bmatrix}
	=
	\begin{bmatrix}
		\alpha_{0}\\
		\alpha_{1}\\
		\alpha_{2}\\
		\alpha_{3}
	\end{bmatrix}	
\end{equation}

Now assume that $\alpha_{j}: j=0,1,2,3$ be distinct positive numbers such that $\Sigma_{k=0}^{3}\alpha_{k}=1$.
\textbf{As in the general case ,an obstruction for algebraic regularity of a probability measure $\mu=\Sigma_{k=o}^{3}\alpha_{k}\delta_{g_{k}}$} is $\alpha_{0}> \frac{4}{5}$.
\end{example}	

\end{remark}	
\end{example}	
\section{Embedding $P(G)$ in Regular Semigroups}	\label{sec4}
Our next goal  is to embed $P(G)$ in larger semigroups
in an optimal way. To do this we use non-commutative Fourier transform techniques.  

\subsection{Non-Commutative Fourier Transforms}
Recall that  for a locally compact topological group $G$	,$\widehat{G}$ will denote the unitary dual space of $	G$. More explicitly
\begin{equation}
	\widehat{G}=\{(\pi,H_{\pi})\}	
\end{equation}

where $\pi:G\rightarrow B(H_{\pi}),{\pi}$  is unitary, irreducible representation of $G$ on a complex separable Hilbert space $H_{\pi}$ with the identification by unitary equivalence of representations.	It is also well-known that when $G$ is compact, then each $H_{\pi}$	is finite dimensional. That means the dimension  $d_{\pi}$ of $H_{\pi}$ is finite and $d_{\pi}=1$ if $G$ is abelian.
The Fourier transform of a $\mu \in P(G)$ is defined as a function on $\hat{\mu}:\widehat{G}\rightarrow B(H_{\pi})$ defined by

\begin{equation}
	\hat{\mu}(\pi)\psi=\int_{G} \pi(g^{-1})\psi \mu(dg)
\end{equation}	
$\pi \in \widehat{G}$.
For a compact, Hausdorff  group $G$ let $	\mathcal{M}=\cup_{d_{\pi}}	M_{d\pi}(\mathbb{C}).$

A map $\Phi:\widehat{G} \rightarrow \mathcal{M}(\widehat{G})$ is called \textbf{Compatible} if  for each $\pi \in \widehat{G}$, $\Phi(\pi)\in M_{d_{\pi}}(\mathbb{C})$. Here $M_{d_{\pi}}(\mathbb{C})$ denotes the set of all $d_{\pi}\times d_{\pi}$ complex matrices after identifying with $B(H_{\pi})$ for each ${\pi}$.

Recall that the set 
\begin{equation}
\tilde{S}(G)= \{\gamma:\widehat{G} \rightarrow \cup_{\pi}M_{d_{\pi}}(\mathbb{C})\}
\end{equation}
is a regular semigroup and let 
\begin{equation}
\Delta(G)	=\{\gamma:\widehat{G} \rightarrow \cup_{\pi}M_{d_{\pi}}(\mathbb{C}),\gamma \quad  compatible.\}
\end{equation}	
\begin{itemize}	
\item[{[1]}]
The problem under	investigation is the algebraic regularity of the following semi-groups and finding the maximal regular subsemigroup of  $ \widehat{P(G)}$.

Since the non-commutative Fourier transform is an isomorphism, it is clear that $\widehat{P(G)}$ is not algebraically regular.

\item[{[2]}]The regularity of the associated semi group $\tilde{S}(G)$.

\item[{[3]}] The regularity of the semigroup $\Delta(G)$.
Observe that these semi-groups are related as follows.
\[
\widehat{P(G)}	\subset \Delta(G)  \subset \tilde{S}(G).
\]
\end{itemize}
\begin{theorem}	Let $G$ be a compact topological group Then $\tilde{S}(G)$ and $\Delta(G)$
are regular semigroups.
\end{theorem}
\begin{proof}
It is well known that $\tilde(S)(G)$ and $\Delta(G)$ are semi-groups. In either case 
regularity is easy to establish, as shown below.
Let $\gamma \in \tilde{S}(G)$  (or $\Delta(G)$). For each $\pi \in \hat{G}$ let
$\gamma^{\dagger}(\pi)$ be  the Moore-Penrose inverse of $\gamma(\pi).$ Clearly $\gamma^{\dagger}(\pi) \in \mathcal{M}(\hat{G})$. If $\gamma \in \Delta(G)$ so is $\gamma^{\dagger}$.
\end{proof}
\section{Minimal Regular Semigroups Containing $P(G)$}\label{sec5}
Next we consider the problem whether there are regular semigroups $\widetilde {\widehat{\Delta(G)}}$ such that
\begin{equation}
\widehat{P(G)}	\subset \widetilde{\widehat{\Delta(G)}} \subset \Delta(G).
\end{equation}	
We restrict our attention to \textbf {compact Lie groups $G$} where new techniques  such as \textbf{Log-Ng positivity} \cite{DAM}  are available which is defined as follows:

\begin{definition}A compatible function $\gamma:\hat{G}\rightarrow M$ is called \textbf{Lo-Ng positive } if 

\begin{equation}
\Sigma_{\pi \in \Omega}	 d_{\pi}tr(\pi(g)\gamma (\pi) B(\pi)) \geq 0
\end{equation}	
whenever 
\begin{equation}
\Sigma_{\pi \in \Omega}	 d_{\pi}tr(\pi(g)B(\pi)) \geq 0
\end{equation}
for all $g\in G$.
\end{definition}
\begin{theorem}(Theorem 4.3.2, The Lo-Ng Criterion\cite{DAM})
Let $P(G)$ denote the class of regular probability measures on a compact Lie group$G$ and 
$\gamma:\widehat{G}\rightarrow \textbf{M}(G)$ be a comptible mapping.Then $\gamma=\hat{\mu}$ if and only if  $\gamma$ is Lo-Ng  positive namely 
\begin{equation}
h_{n}(g)=\Sigma_{\pi\in S_{n}}z^{(n)}_{\pi}d_{\pi}	tr(\pi(g)\gamma(\pi)) \geq 0
\end{equation}	
for all $g \in G$,where $ \#(S_{m}),\#(S_{n}) <\infty$ if $m<n$ and $\pi_{0}\in S_{n}$ for all $n$.
\end{theorem}

\begin{remark}

	\begin{itemize}
\item[{[1]}]The above theorem is a non-commutative analogue of the celebrated Bochkner's theorem:
Let $G$ be a locally compact abelian group and $\hat{G}$ be the dual group of characters. Let $F:\hat{G}\rightarrow \mathbb{C}$. Then $F$ is the Fourier transform of a measure $\mu,$

\begin{equation}\nonumber
\hat{\mu}(\chi)=\int_{G}\bar{\chi(g)}\mu(dg) F(x_{i}-x_{j}) \geq 0 \textrm{ if and only if } F(\hat{e})=1, \,F \textrm{ is continuous at }\hat{e}.
\end{equation}	

\begin{equation}\nonumber
\hat{\mu}(\pi)\psi=\int_{G} \pi(g^{-1})\psi \mu(dg),\, \pi \in \hat{G}.
\end{equation}	
\item[{[2]}]Therefore  $\hat{\mu}^{\dagger}$ is Lo-Ng positive ,then $\mu$ has a generalised inverse in $P(G)$.
\end{itemize}
\end{remark}
\section{A few more related questions }\label{sec6}
Let $\Omega(G)$ be the semi-group of all finite products of idempotents in $P(G)$. There are two questions associated with this.
\begin{itemize}
	
\item[{[1]}]	Is $\Omega(G)$ regular?. If so

\item[{[2]}] Is $\Omega(G)$ the maximal regular semigroup contained in $P(G)$ ?.

\item[{[3]}]What are the regular elements in $P(G)$?
\end{itemize}

        \textbf{Acknowledgement:} The author is thankful to KSCSTE, Government Of Kerala, for financial support by awarding Emeritus Scientist Fellowship, during which a major part of this work was done. Also, a part of this research work was presented to the international conference, ICSAOT-22, 28-31, March 2022, held at the Department Of Mathematics, CUSAT, in honour of Prof. P.G. Romeo.

	\nocite{*}
	\bibliographystyle{amsplain}
	\bibliography{Reqularity}
	
	\end{document}